\documentclass{amsart}
\usepackage{amssymb}
\usepackage{amsmath}
\usepackage{amsthm}
\usepackage{enumerate}

\makeatletter
\@namedef{subjclassname@2010}{%
  \textup{2010} Mathematics Subject Classification}
\makeatother

\numberwithin{equation}{section}

\newtheorem{theorem}{Theorem}[section]
\newtheorem{corollary}[theorem]{Corollary}

\newtheorem{lemma}[theorem]{Lemma}

\newcommand{\fktvec}{\Big(f_k(t)\Big)_{k=0}^{n-1}}
\newcommand{\gktvec}{\Big(g_k(t)\Big)_{k=0}^{n-1}}
\newcommand{\hktvec}{\Big(h_k(t)\Big)_{k=0}^{n-1}}

\begin{document}

\baselineskip=17pt

\title[Rosenthal inequalities for quasi-Banach spaces]{Rosenthal's inequalities: $\Delta-$norms and quasi-Banach symmetric sequence spaces}

\author[]{Yong Jiao}
\address{School of Mathematics and Statistics, Central South University, Changsha 410075, People's Republic of China}
\email{jiaoyong@csu.edu.cn}

\author[]{Fedor Sukochev}
\address{School of Mathematics and Statistics, University of NSW, Sydney,  2052, Australia}
\email{f.sukochev@unsw.edu.au}

\author[]{Guangheng Xie}
\address{School of Mathematics and Statistics, Central South University, Changsha 410075, People's Republic of China}
\email{xieguangheng@csu.edu.cn}

\author[]{Dmitriy  Zanin}
\address{School of Mathematics and Statistics, University of NSW, Sydney,  2052, Australia}
\email{d.zanin@unsw.edu.au}

\thanks{Yong Jiao is supported by (No.11471337, No.11722114); Fedor Sukochev and Dmitriy Zanin are supported by the Australian Research Council.}

\begin{abstract}
Let $X$ be a symmetric quasi-Banach function space with Fatou property and let $E$ be an arbitrary symmetric quasi-Banach sequence space. Suppose that $(f_k)_{k\geq0}\subset X$ is a sequence of independent random variables. We present a necessary and sufficient condition on $X$ such that the quantity
$$\Big\|\ \Big\|\sum_{k=0}^nf_ke_k\Big\|_{E}\ \Big\|_X$$
admits an equivalent characterization in terms of disjoint copies of $(f_k)_{k=0}^n$ for every $n\ge 0$; in particular, we obtain the deterministic description of
$$\Big\|\ \Big\|\sum _{k=0}^nf_ke_k\Big\|_{\ell_q}\ \Big\|_{L_p}$$
for all $0<p,q<\infty,$ which is the ultimate form of Rosenthal's inequality. We also consider the case of a $\Delta$-normed symmetric function space $X$, defined via an Orlicz function $\Phi$ satisfying the $\Delta_2$-condition.
That is, we provide a formula for \lq\lq $E$-valued $\Phi$-moments\rq\rq, namely the quantity $\mathbb{E}\big(\Phi\big(\big\|(f_k)_{k\geq0} \big\|_E \big)\big)$, in terms of the sum of disjoint copies of $f_k, k\geq0.$
\end{abstract}



\maketitle

\section{Introduction}

Let  $E$ be a symmetric quasi-Banach sequence space, let $(e_k)_{k=0}^\infty\subset E$ be the
standard basic sequence in $E$, that is $e_k:=(\underbrace{0,0,\dots,0}_{k-1},1,0,0,\dots)$  and let $X$ be a quasi-Banach symmetric function space or a symmetric $\Delta$-normed function space on
$[0,1]$ (for all unexplained terms here see Section
\ref{Preliminaries} below). For an arbitrary sequence of independent
functions $\{f_k\}_{k=0}^n\subset X$, $n\ge 0$, we consider the
quantity
\begin{equation}\label{main quantity}
\Big\|\  \Big\|\sum_{k=0}^nf_ke_k\Big\|_{E}\ \Big\|_X.
\end{equation}

The main objective of this paper is to find an equivalent deterministic estimate for this quantity in terms of disjoint copies of
functions from the given sequence.  We now briefly outline key results concerning the study of \eqref{main quantity} for special cases of $E$ and $X$ in the existing literature.

The origin of intensive studies concerning  quantity \eqref{main quantity} may be found in a famous paper of Rosenthal \cite{Rosenthal} (see also its detailed account in \cite[Theorem 6]{Astashkin2010}), which treated the case $E=\ell_1$ and $X=L_p,$ $2<p<\infty.$ Next, in the special case when $E$ is a symmetric (Banach) Orlicz sequence space $\ell_N$, $X=L_1(0,1)$ and when the sequence $(f_k)_{k=0}^n$ is a sequence of identically distributed random variables the study of \eqref{main quantity} was initiated by Gordon, Litvak, Sch\"utt and Werner \cite{Gordon} who proved that
\begin{equation}\label{GLSW}
\mathbb E(\|(a_kf_k)\|_{\ell_N})\approx_N \|(a_i)\|_{\ell_\Phi}
\end{equation}
where $\Phi$ is an Orlicz function depending on $N$ and the distribution function of $f_1$
(here $A\approx_N B$ means that the ratio $A/B$ is bounded below and above by constants which depend only on $N$). Moreover, the generalizations of \eqref{GLSW} to Musielak-Orlicz norms and tensor products of random variables were studied in \cite{ACPP} and \cite{APP}, respectively. Further, answering Yehoram Gordon's question whether a formula similar  to \eqref{GLSW} exists for arbitrary sequences of independent random variables and not just for  scalar multiples of independent identically distributed random variables, Montgomery-Smith \cite{Montgomery-Smith} produced a positive answer in the more general setting of symmetric Banach function and Orlicz sequence spaces.
Let $X$ be a symmetric Banach function space on $[0,1],$ let $E=\ell_N$ be an Orlicz sequence space and let $(f_k)_{k=0}^n\subset X,$ $n\ge 1$ be a sequence of independent random variables.
The notation $f=:\bigoplus_{k=0}^\infty f_k$ stands for a disjoint sum of random
variables $(f_k)$ considered as a measurable function on $(0,\infty).$ \cite[Theorem 1]{Montgomery-Smith}
asserts that if $L_q\subset X$ for some $1\leq q<\infty,$ then
\begin{equation}\label{m}
\big\|\|(f_k)_{k\geq0}\|_E\big\|_X\approx_{X,E} \big\|\mu(f)\chi_{(0,1)}\big\|_X+\big\|\big(\mu(k,f)\big)_{k\geq1}\big\|_E.
\end{equation}
Here, $\mu(f)$ denotes the decreasing rearrangement function of $f.$ The proof given in \cite{Montgomery-Smith} has been based on a deep and detailed development of Rosenthal's inequality due to Carothers and Dilworth \cite{CD1} who treated the setting when $X$ is a Lorentz space $L_{p,q},$ $1\leq p<\infty,$ $1\leq q\leq \infty,$ and especially to Johnson and Schechtman \cite{JS} who managed to thoroughly investigate the case when $X$ is an arbitrary quasi-Banach symmetric function space such that $L_q\subset X\subset L_p$ for some $0<q<p<\infty$ and $E=\ell_1.$  In all these works, a crucial component of the proof is an application of a well-known inequality by Hoffman-J\o rgensen \cite{HJ} or its versions. The fact that the Johnson and Schechtman's result \cite{JS} initially proved for $E=\ell_1$ can be easily extended to the setting when $E=\ell_p$,
$1\leq p<\infty$, was noted in \cite{Montgomery-Smith}, see also a detailed proof in \cite[Theorem 6.7]{AS2005}.

An important case, when $X=L_p, 1\leq p<\infty$ and $E=\ell_{1,\infty}$ (a quasi-Banach weak $\ell_1$-space) was treated by Junge in \cite{Junge}.  The approach based on Rosenthal and Johnson-Schechtman inequalities and their ramifications has culminated in \cite[Theorem 1]{Astashkin sequence} which provides necessary and sufficient conditions on the symmetric Banach space $X$ (in terms of the so-called Kruglov operator $K$ introduced in \cite{AS2005})  under which the estimates \eqref{m} hold. This result, however, was established under two rather restrictive conditions on the quasi-Banach symmetric sequence space $E$. Let us assume (without loss of generality) that $\|e_k\|_E=1,$ $k\ge 0.$ For every $k\in\mathbb{N},$ define a dilation operator $\sigma_k:l_{\infty}\to l_{\infty}$ as follows: if $(a_j)_{j\ge 0}$ is a bounded sequence, then
$$
 \sigma_k((a_j)_{j\ge 0}):=(\underbrace{a_0,\cdots,a_0}_{\mbox{$k$ times}},
 \underbrace{a_1,\cdots,a_1}_{\mbox{$k$ times}},\cdots, \underbrace{a_j,\cdots,a_j}_{\mbox{$k$ times}},\cdots).
$$
The assumptions on $E$ made in \cite{Astashkin sequence} are as follows:
$${\rm(i)} \quad \|\sigma_k\|_{E\to E}\leq k\mbox{ for all }k\ge 1\mbox{ and } {\rm (ii)} \quad \ell_1\subset E.$$
Observe that in the case when $E$ is a Banach symmetric sequence space the latter assumptions are satisfied automatically. However, in the important case $E=\ell_q$, $0<q<1$ both conditions above fail and the question whether equivalence \eqref{m} holds even for the case $X=L_p,\ 0<p<\infty$ remains unanswered.

One of the main objectives of this paper is to provide a complete answer to this question.
In fact, we shall present necessary and sufficient conditions on a quasi-Banach symmetric function space $X$
such that the deterministic equivalence \eqref{m} holds for every quasi-Banach symmetric sequence space $E$.
This result strengthens and complements all above mentioned papers \cite{Rosenthal, Gordon, CD1, JS, Montgomery-Smith, Junge, AS2005, Astashkin sequence,ACPP,pacific,APP}. We also mention that Astashkin and Tikhomirov \cite[Corollary 2]{Astashkin JUNGE} established that the inequality
$$\big\|\|(f_k)_{k\geq0}\|_E\big\|_X\gtrsim_{X,E}\big\|\big(\mu(k,f)\big)_{k\geq1}\big\|_E,\quad  f=\bigoplus_{k=0}^n f_k.$$
holds for all symmetric quasi-Banach spaces $X$ and $E.$
The converse inequality was not studied in \cite{Astashkin JUNGE} and it follows from the results of this paper,
that there are examples of separable symmetric spaces $X$ where the converse inequality fails.

We refer the reader for a more detailed account of all these developments to \cite{Astashkin2010} and just recapitulate our main point: the results of Johnson and Schechtman \cite{JS}, Astashkin and Sukochev \cite{{Astashkin sequence}, Astashkin2010} and Astashkin and Tikhomirov \cite{Astashkin JUNGE}
necessitated that the next chapter of the studies of the quantity \eqref{main quantity}
should be done for the case when $X$ and $E$ are not necessarily Banach but quasi-Banach or even $\Delta$-normed spaces.
The study of the case when $X$ is a $\Delta$-normed space, has been recently initiated in  \cite{JSXZ},
where the authors considered the problem of computing the $\Delta$-norm given by $\|f\|_{\Phi}=\mathbb{E}\big(\Phi(|f|)\big)$,
where $\Phi$ is an Orlicz function. In particular, \cite[Theorem 1.3]{JSXZ} asserts that
\begin{equation}\label{jsxz-classical}
\mathbb{E}\Big(\Phi(\sum_{k=0}^nf_k)\Big)\approx_{{\Phi}}\mathbb{E}\Big(\Phi\big(\mu(f)\chi_{[0,1]}\big)\Big)+
\Phi\big(\|f\|_1\big),\quad f=\bigoplus_{k=0}^n f_k,
\end{equation}
for an Orlicz function $\Phi$ satisfying the $\Delta_2$-condition and
for an arbitrary sequence $(f_k)_{k=0}^n\subset L_{\Phi}[0,1]$ ($n\in\mathbb{N}$) of non-negative independent random variables.

We are now ready to state our main results. See section 2 for the definition of Kruglov operator $K$ and any other unexplained notations.

\begin{theorem}\label{main-new} Let $X$ be a quasi-Banach symmetric function space with Fatou property. The following conditions are equivalent
\begin{enumerate}[{\rm (i)}]
\item\label{mainnewa} for every symmetric quasi-Banach sequence space $E$ and for an arbitrary sequence of independent random variables $(f_k)_{k\geq0}\subset X,$ we have
\begin{equation}\label{main}
\big\|\|(f_k)_{k\geq0}\|_E\big\|_X\approx_{X,E} \big\|\mu(f)\chi_{(0,1)}\big\|_X+\big\|\big(\mu(k,f)\big)_{k\geq1}\big\|_E,\quad f=\bigoplus_{k\geq0}f_k.
\end{equation}
\item\label{mainnewb} for every $0<p<1$ we have $K:X_p\to X_p,$ where $X_p$ is the $p$-th power of $X.$
\end{enumerate}
\end{theorem}

The condition \eqref{mainnewb} above is in sharp contrast with the results of \cite{Astashkin sequence}, where only condition $K:X\to X$ is used.

Note that the Fatou property of the space $X$ is only used when we derive \eqref{mainnewb} from \eqref{mainnewa}. In this direction, the argument is relatively short (see the very end of the proof of Theorem \ref{main-new}). The proof of the reverse implication (which occupies the bulk of the proof) perfectly works for every quasi-Banach space $X$ without additional assumptions (e.g. without Fatou property).

We immediately deduce the following ultimate form of Rosenthal's inequality
and nicely finalize the story started in \cite{Rosenthal}.

\begin{corollary}\label{lplq corollary}
Let $(f_k)_{k\geq0}\subset L_p(0,1)$ be an arbitrary sequence of independent random variables.
\begin{enumerate}[{\rm (1)}]
\item For all $0<p\leq q<\infty,$ we have
$$\Big\|\ \Big\|\sum _{k\geq0}f_ke_k\Big\|_{\ell_q}\ \Big\|_{L_p}\approx_{p,q} \big\|f\big\|_{L_p+L_q},\quad f=\bigoplus_{k\geq0}f_k.$$
\item For all $0<q\leq p<\infty,$ we have
$$\Big\|\ \Big\|\sum _{k\geq0}f_ke_k\Big\|_{\ell_q}\ \Big\|_{L_p}\approx_{p,q} \big\|f\big\|_{L_p\cap L_q},\quad f=\bigoplus_{k\geq0}f_k.$$
\end{enumerate}
\end{corollary}

We have the following $\Delta$-normed estimates extending \eqref{jsxz-classical} to an arbitrary symmetric quasi-Banach sequence space.
\begin{theorem}\label{main modular thm}
Let $\Phi\in\Delta_2$ be an Orlicz function. Let $E$ be a quasi-Banach symmetric sequence space. For every $n\in\mathbb{N}$ and for every sequence $(f_k)_{k=0}^{n-1} \subset L_{\Phi}(0,1)$ of independent random variables and $f=\bigoplus_{k=0}^{n-1}f_k$, we have
\begin{equation}\label{first main}
\int_0^1\Phi\Big(\Big\|\fktvec\Big\|_E\Big)dt\approx_{E,\Phi}\,\int_0^1\Phi(\mu(t,f))dt
+\Phi\Big(\Big\|\Big(\mu(k,f)\Big)_{k=1}^n\Big\|_E\Big).
\end{equation}
\end{theorem}

The $\Delta_2-$condition in Theorem \ref{main modular thm} is unavoidable as it guarantees the linearity of the modular space. On the other hand, this condition plays a role similar to the boundedness of the Kruglov operator in Theorem \ref{main-new}. This manifests how Rosenthal inequalities in the setting of modular spaces differ from those in the setting of quasi-Banach symmetric spaces. Theorem \ref{main modular thm} nicely finalizes the story initiated by Gordon, Litvak, Sch\"utt and Werner \cite{Gordon}.


Our proof of Theorem \ref{main modular thm} borrows some ideas from \cite{Junge} and \cite{Jiao Sukochev and Zanin}. It should also be pointed out that we do not assume conditions (i) and (ii) as in \cite{{Astashkin sequence}, Astashkin2010} (see above), which shows substantial difference between our current approach and the one in \cite{AS2005, {Astashkin sequence}, Astashkin2010}.

The rest of the paper is organized as follows. In Section \ref{Preliminaries},
we recall some notions and notation on symmetric quasi-Banach function/sequence spaces,
Orlicz functions and Kruglov operators. Then we present some known results that will be used to prove Theorem \ref{main-new}. In Section \ref{pf of thm1}, we present the proofs of Theorem \ref{main-new} and Corollary \ref{lplq corollary}. Finally, in the last section, we prove Theorem \ref{main modular thm}.

\section{Preliminaries}\label{Preliminaries}

In this section, we provide some preliminaries which are necessary for the whole paper.
\subsection{Symmetric quasi-Banach function/sequence spaces}

For a measurable function $f$ on $(0,1)$ or on $(0,\infty)$ (equipped with the Lebesgue measure $m$), we define a distribution function by setting
$$d_f(s):=m(\{t:f(t)>s\}),\quad s\in\mathbb{R}.$$

Let $S(0,1)$ denote the space of all Lebesgue measurable functions on $(0,1).$ Respectively,
$$S(0,\infty)=\Big\{f:(0,\infty)\to\mathbb{R}:\ f\mbox{ is measurable and }d_{|f|}(s)<\infty\mbox{ for some }s>0\Big\}.$$
For every $f\in S(0,1)$ (or $f\in S(0,\infty)$), its decreasing rearrangement $\mu(f)$ (strictly speaking, this is the decreasing rearrangement of $|f|$) is defined by the formula
$$\mu(t,f):=\inf\{\lambda\geq0:d_{|f|}(\lambda)<t\},\quad t>0.$$

A linear space over $\mathbb{C}$ is called quasi-normed if it is equipped with the functional $\|\cdot\|_E:E\to\mathbb{R}_+$ such that
\begin{enumerate}[{\rm (1)}]
\item $\|x\|_E=0$ if and only if $x=0.$
\item for every $x\in E$ and for every $\alpha\in\mathbb{C}$ we have $\|\alpha x\|_E=|\alpha|\|x\|_E$
\item for every $x,y\in E,$ we have
$\|x+y\|_E\leq C_E(\|x\|_E+\|y\|_E),$ where $C_E$ is a positive constant depending only on $E.$
\end{enumerate}
Convergent sequences and Cauchy sequences are defined exactly as in the normed case. A quasi-normed space is called quasi-Banach if every Cauchy sequence converges. The constant $C_E$ is called the concavity modulus of the space $E.$

Let $(E,\|\cdot\|_E)$ be a quasi-Banach space. $E$ is said to be a quasi-Banach function space on $(0,1)$ if $E\subset S(0,1),$ that is, $E$ consists of measurable functions on $(0,1).$ A quasi-Banach function space $(E,\|\cdot\|_E)$ is called a quasi-Banach lattice if, from $f\in E,$ $g\in S(0,1)$ and $|g|\leq|f|$, it follows that $g\in E $ and $\|g\|_E \leq\|f\|_E.$ A quasi-Banach lattice $E$ is said to be symmetric quasi-Banach function space if, for every $f\in E$ and for every measurable function $g$, the assumption $\mu(g)=\mu(f)$ implies that $g\in E$ and $\|g\|_E=\|f\|_E.$  The latter notion admits a natural extension to symmetric $\Delta$-normed function spaces whose definition we now recall.
Let $\Omega$ be a linear space over the field $\mathbb{C}$.
A function $\|\cdot\|$ from $\Omega$ to $\mathbb{R}$ is a $\Delta$-norm, if for all $x,y \in \Omega$ the following properties hold:
\begin{enumerate}[{\rm (1)}]
\item $\|x\| \geqslant 0$; $\|x\| = 0 \Leftrightarrow x=0$;
\item $\|\alpha x\| \leqslant \|x\|$ for all $|\alpha| \le1$;
\item $\lim _{\alpha \rightarrow 0}\|\alpha x\| = 0$;
\item there exists a constant $C_\Omega\geq 1$ such that
$$\|x+y\| \le C_\Omega \cdot (\|x\|+\|y\|),\quad x,y\in\Omega.$$
\end{enumerate}

The couple $(\Omega, \|\cdot\|)$ is called a $\Delta$-normed space.
We note that the definition of a $\Delta$-norm given above is the same as the one given in \cite{Kalton1985}.
It is well-known that every $\Delta$-normed space $(\Omega,\|\cdot\|)$ is metrizable and conversely every (translation invariant) metrizable space can be equipped with a $\Delta$-norm (see e.g.\cite{Kalton1985}).
Note that properties $(2)$ and $(4)$ of a $\Delta$-norm imply that for any $\alpha\in\mathbb{C}$,
there exists a positive constant $M$ such that $\|\alpha x\|\leq M\|x\|,\, x\in \Omega$, in particular,
if $\|x_n\|\to 0, (x_n)_{n=1}^\infty\subset \Omega$, then $\|\alpha x_n\|\to 0$. If $(\Omega, \|\cdot\|)$ a $\Delta$-normed space of functions (say on $(0,1)$) such that $f\in \Omega$ and $g\in S(0,1)$ with $\mu(f)\ge \mu(g)$ imply $g\in \Omega$ and $\|f\|\ge \|g\|$, then we call $(\Omega, \|\cdot\|)$ a symmetric $\Delta$-normed space.

We say that a symmetric quasi-Banach function space $E$ satisfies the Fatou property if, for every bounded sequence $(x_n)_{n\geq0}\subset E,$ the convergence $x_n\to x$ almost everywhere implies that $x\in E$ and
$$\|x\|_E\leq\liminf_{n\to\infty}\|x_n\|_E.$$

Given $0<p<\infty,$ the $p$-th power of the quasi-Banach symmetric space $X$ is defined by setting
$$X_p:=\{f\in S(0,1):|f|^{\frac{1}{p}}\in X\}$$
and
$$\|f\|_{X_p}:=\||f|^{\frac{1}{p}}\|^p_X.$$
Using that $\mu(|f|^p)=\mu(f)^p$ for any $f\in S(0,1),$ one can see that $X_p$ is symmetric if $X$ is symmetric. It is also a simple fact that if $X$ satisfies the Fatou property (in the sense of \cite[page 260]{pacific}), then $X_p$ also satisfies the Fatou property for every $0<p<\infty.$ We refer to \cite[Chapter 2]{OR} and references therein for more details on $p$-th power of a quasi-Banach function space.

If $\xi=(\xi_n)_{n=0}^{\infty}$ is a bounded sequence of real numbers, then its \emph{distribution
function} $d_{\xi}$ is defined by setting, for any $t\in\mathbb{R}$,
$$d_{\xi}(t):={\rm Card}(\{n\geq0:\ \xi_n>t\}),$$
where for every $A\subset\mathbb{Z}_+,$ ${\rm Card}(A)$ is the cardinality of $A.$ Then for any $n\geq0$,
$$\mu(n,\xi)=\inf\{\lambda\geq0:\ d_{|\xi|}(\lambda)\le n\}.$$
Then for two sequences of nonnegative numbers $a:=(a_n)_{n=0}^{\infty}$ and $b:=(b_n)_{n=0}^{\infty},$ $d_a\le d_b$ implies that $\mu(a)\le \mu(b).$

A quasi-Banach sequence space $E$ is said to be symmetric if from the assumptions $a\in E$ and $\mu(b)\leq\mu(a)$ it follows that $b\in E $ and $\|b\|_E\leq \|a\|_E.$ Without loss of generality we will assume throughout that $\|e_k\|_E=1$ $(k=0,1,2,\ldots),$ where the vectors $e_k$ are the vectors of the standard basis in sequence spaces.

Let $f_k$, $k\geq0$, be elements from $S(0,1)$ and let $g_k\in S(0,\infty)$, $k\geq0$, be their disjoint copies; that is, $f_k$ and $g_k$ are identically distributed random variables for all $k\geq0$, and $g_lg_m=0$ if $l\neq m$. For example, we can set $g_k(t)=f_k(t-k)\chi_{[k,k+1)}(t),$ $t>0.$ For the function $\sum_{k\geq0}g_k$, which is frequently called the disjoint sum of $f_k,k\geq0$, we shall use the suggestive notation $\bigoplus_{k\geq0}f_k$. It is important to observe that the distribution function of a disjoint sum $\bigoplus_{k\geq0}f_k$
does not depend on the particular choice of elements $g_k,$ $k\geq0.$ Note the obvious equality
$$d_{\bigoplus_{k=0}^n f_k}=\sum_{k=0}^nd_{f_k}.$$
In the special case when $\sum_{k=0}^n\mathbb{P}({\rm supp}(f_k))\leq1,$ it is convenient to view the sum $\bigoplus_{k=0}^nf_k$ as a measurable function on $(0,1)$.

We recall that the dilation operator $\sigma_s:S(0,1)\rightarrow S(0,1),$ $s\in(0,1)$ is given by $(\sigma_sx)(t)=x(\frac{t}{s})$ if $t\in(0,s);$ otherwise $(\sigma_sx)(t)=0.$

\subsection{Orlicz functions and Kruglov operators}

Let $\Phi$ be an Orlicz function on $[0,\infty)$, i.e. a continuous
increasing and convex function satisfying $\Phi(0)=0$ and $\lim_{t\rightarrow\infty}\Phi(t)=\infty$.
Recall that an Orlicz function $\Phi$ on $[0,\infty)$ satisfies the $\Delta_2$-condition if there is a positive constant $C$ such that  $\Phi (2t) \leq C\Phi (t) $ for all $t>0$. In this case, we write $\Phi\in\Delta_2$. Let $1\leq p\leq q\leq\infty,$ an Orlicz function $\Phi$ is said to be $p$-convex if the function $t\rightarrow\Phi(t^{\frac{1}{p}}),$ $t>0$ is convex, and $\Phi$ is said to be $q$-concave if the function $t\rightarrow\Phi(t^{\frac{1}{q}}),$ $t>0$ is concave. An Orlicz function $\Phi\in\Delta_2$ if and only if it is equivalent to a $q$-concave Orlicz function for some $q<\infty,$ (see, for instance, \cite[Lemma 5]{AS2014JMAA}).

By $L_{\Phi}$ (the Orlicz function space associated with $\Phi$) we denote the class of all measurable functions $f$ on $(0,1)$ (or on $(0,\infty)$) such that the norm
$$\|f\|_{L_{\Phi}}=\inf\Big\{\lambda>0:\int_0^{\infty}\Phi\Big(\frac{|f(t)|}{\lambda}\Big)dt\leq1\Big\}$$
is finite. It is well known that $L_\Phi$ is a symmetric function space. On the other hand, the space $L_{\Phi}$ equipped with modular $\|f\|_{\Phi}=\int_0^{\infty}\Phi({|f(t)|})dt$ is an example of a $\Delta$-normed space, provided $\Phi\in\Delta_2$ (see \cite[pp.28-29]{Kalton1985}).

Before introducing the definition of the Kruglov operator originated in \cite{AS2005} (see also \cite{pacific}), we consider the probability product space
$$(\Omega,\mathbb{P}):=\prod_{k=0}^{\infty}((0,1),\mathbb{P}_k),$$
($\mathbb P_k$ is the Lebesgue measure on $(0,1),k\geq0$). Observe that in an arbitrary symmetric space, the norms of any two elements with identical distribution coincide. Hence, using a bijective measure-preserving transformation between measure space $(\Omega,\mathbb{P})$ and $((0,1),\mathbb P)$, we  identify an arbitrary measurable function $f(\omega)=f(\omega_0,\omega_1,\cdots,\omega_n\cdots)$ on $(\Omega,\mathbb{P})$ with the corresponding element from $S(0,1)$. A particular form of the measure-preserving transformation used in such identification does not play any role and we completely suppress it from the notations. Thus, we  view the set $\Omega$ as $(0,1)$ and any measurable function on $(\Omega,\mathbb{P})$ as a function from $S(0,1)$.

Now, we are ready to introduce the notation of the Kruglov operator. Let $(A_n)_{n=0}^{\infty}$ be a fixed sequence of mutually disjoint measurable subsets of $(0,1)$ such that $\mathbb P(A_n)=\frac1{e\cdot n!}.$ Define the operator $K:S(0,1)\rightarrow S(0,1)$
by setting
\begin{equation}\label{K def}
Kf(\omega):=\sum_{n=1}^{\infty}\sum_{k=1}^n f(\omega _k)\chi_{A_n}(\omega_0).
\end{equation}

We end this section by introducing two useful lemmas which will be needed for the proof of Theorem \ref{main-new}. The following lemma was proved in \cite[Theorem 7]{pacific}.

\begin{lemma}\label{lem-pjm}
Let $X$ and $Y$ be quasi-Banach symmetric spaces on $(0,1)$ and let $Y$
have the Fatou property. Suppose that there exists a positive constant $C$ such
that for every sequence of nonnegative independent random variables $(f_k)_{k=1}^n\subset X$, $n\in\mathbb N$,
with $\sum_{k=1}^{n}m({\rm supp}(f_k))\leq 1,$ we have
$$\left\|\sum_{k=1}^nf_k\right\|_Y\le C \left\|\bigoplus_{k=1}^{n}f_k\right\|_X.$$
Then the operator $K$ maps $X$ into $Y$ and $\|K\|_{X\to Y}\le C$.

The assertion remains valid under the assumption that the above inequality holds
for $X=Y$, where $X$ is a separable quasi-Banach symmetric space.
\end{lemma}

The following technical lemma comes from Junge \cite[Theorem 0.3]{Junge}.

\begin{lemma}\label{lem-junge}
There exists a positive constant $c_0$ with the following property. Let $(p_{ij})_{ij}$ be a doubly stochastic matrix, $i.e.$ for all i, j,
$$\sum_{k}p_{ik}=1=\sum_{k}p_{kj}.$$
Then
$$\left[\sum_{j_1,\ldots,j_n=1}^n\left(\sup_r\frac1r{\rm Card}\left\{i:\ j_i\le r\right\}\right)^p\prod_{k=1}^np_{kj_k}\right]^{\frac1p}
\le c_0\frac{p}{1+\log p}.$$
The order of growth is optimal.
\end{lemma}

\section{Proofs of Theorem \ref{main-new} and Corollary \ref{lplq corollary}}\label{pf of thm1}
In this section we establish Theorem \ref{main-new} and Corollary \ref{lplq corollary}. The proof of the implication of (ii)$\Longrightarrow$(i) of Theorem \ref{main-new} is divided into two parts: the upper estimate and the lower estimate. Lemmas \ref{trivial sigma lemma}--\ref{head lemma} are needed in the proof of the upper estimate in \eqref{main}. One of the key tools used in this part is the combinatorial estimate obtained in \cite{Junge} (see Lemma \ref{lem-junge}). Lemma \ref{lower simple} is needed in the proof of the lower estimate in \eqref{main}. By using \cite[Theorem 7]{pacific}, we prove the implication of (i)$\Longrightarrow$(ii). We then present the proof of Corollary \ref{lplq corollary}.

\begin{lemma}\label{trivial sigma lemma} Let $E$ be a quasi-Banach symmetric sequence space and let $a\in E.$ We have
$$\|\sigma_ma\|_E\leq C_E\cdot m^{1+\log_2(C_E)}\cdot\|a\|_E.$$
\end{lemma}
\begin{proof} By induction, we have
$$\|\sum_{k=1}^{2^n}x_k\|_E\leq C_E^n\cdot\sum_{k=1}^{2^n}\|x_k\|_E.$$
For $m\in\mathbb{N},$ there exists a integer $n\geq0$ such that $m\in[2^n,2^{n+1});$ then, setting $x_{m+1}=\cdots=x_{2^{n+1}}=0,$ we obtain
$$\|\sum_{k=1}^mx_k\|_E=\|\sum_{k=1}^{2^{n+1}}x_k\|_E\leq C_E^{n+1}\cdot\sum_{k=1}^{2^{n+1}}\|x_k\|_E=$$
$$=C_E^{n+1}\cdot\sum_{k=1}^m\|x_k\|_E\leq C_E\cdot m^{\log_2(C_E)}\cdot\sum_{k=1}^m\|x_k\|_E.$$
If $x_k,$ $1\leq k\leq m,$ are pairwise disjoint copies of $a,$ then
$$\mu\Big(\sum_{k=1}^mx_k\Big)=\sigma_m\mu(a).$$
Hence,
$$\|\sigma_ma\|_E\leq C_E\cdot m^{\log_2(C_E)}\cdot\sum_{k=1}^m\|x_k\|_E=C_E\cdot m^{1+\log_2(C_E)}\cdot\|a\|_E.$$
\end{proof}

Note that Lemma \ref{lem-junge} plays an important role in the following lemma whose proof should be compared with the proofs of \cite[Theorem 30]{Astashkin2010} and \cite[Theorem 1]{Astashkin sequence}.

\begin{lemma}\label{vector p-estimate} Let $E$ be a quasi-Banach symmetric sequence space and let $(f_k)_{k=0}^{n-1}\subset L_p(0,1)$ be a sequence of independent random variables. For every $p\geq 1,$ we have
$$\Big\|\Big\|(f_k)_{k\geq0}\Big\|_E\Big\|_p\lesssim_E \Big(\frac{p}{\log(ep)}\Big)^{1+\log_2(C_E)}\cdot \Big\|\Big(\mu(k,f)\Big)_{k=0}^{n-1}\Big\|_E,\quad f=\bigoplus_{k=0}^{n-1}f_k.$$
\end{lemma}
\begin{proof} Since $E$ is symmetric, we may assume that $f_0,f_1,\ldots,f_{n-1}$ are nonnegative. Denote, for brevity, $a_k=\mu(k,f),$ $0\leq k<n.$ Without loss of generality, $\mu(f)$ does not have intervals of constancy on the interval $(0,n).$

{\bf Step 1: }For $0\leq k,l\leq n-1$, we set
$$p_{k,l}=m(\{t\in(0,1):\ \mu(l+1,f)<f_k(t)\leq\mu(l,f)\}).$$
We have
$$\sum_{l=0}^{n-1}p_{k,l}=m(\{t\in(0,1): \mu(n,f)<f_k(t)\leq\mu(0,f)\})=1$$
and
$$\sum_{k=0}^{n-1}p_{k,l}=m(\{t\in(0,n):\ \mu(l+1,f)<f(t)\leq\mu(l,f)\})=1.$$
Thus, the matrix $(p_{k,l})_{k,l=0}^{n-1}$ is doubly stochastic.

{\bf Step 2:} We claim that
$$\int_0^1\Big\|\fktvec\Big\|_E^pdt\leq\sum_{{\bf l}\in\Delta_n}\Big\|(a_{l_k})_{k=0}^{n-1}\Big\|_E^p\cdot \prod_{k=0}^{n-1}p_{k,l_k}.$$
Here, $\Delta_n$ is the collection of all maps from $\{0,\cdots,n-1\}$ to itself and the map
$${\bf l}=
\begin{cases}
0\to l_0\\
1\to l_1\\
\vdots\\
n-1\to l_{n-1}
\end{cases}
$$
is identified with the sequence ${\bf l}=(l_0,\cdots,l_{n-1}).$

Indeed, for ${\bf l}\in\Delta_n,$ consider the set
$$A_{{\bf l}}=\{t\in(0,1):\ \mu(l_k+1,f)<f_k(t)\leq \mu(l_k,f),\ 0\leq k<n\}.$$
Since the functions $(f_k)_{k=0}^{n-1}$ are independent, it follows that
$$m(A_{{\bf l}})=\prod_{k=0}^{n-1}p_{k,l_k}.$$
Therefore, we have
$$\int_0^1\Big\|\fktvec\Big\|_E^pdt=\sum_{{\bf l}\in\Delta_n}\int_{A_{{\bf l}}}\Big\|\fktvec\Big\|_E^pdt\leq$$
$$\leq\sum_{{\bf l}\in\Delta_n}\int_{A_{{\bf l}}}\Big\|(a_{l_k})_{k=0}^{n-1}\Big\|_E^pdt=\sum_{{\bf l}\in\Delta_n}\Big\|(a_{l_k})_{k=0}^{n-1}\Big\|_E^p\cdot \prod_{k=0}^{n-1}p_{k,l_k}.$$
This proves the claim of Step 2.

{\bf Step 3:} For every ${\bf l}\in\Delta_n,$ we claim
$$\mu((a_{l_k})_{k=0}^{n-1})\leq\sigma_{C({\bf l})}a,\quad a=(a_k)_{k=0}^{n-1},$$
where
$$C({\bf l})=\Big\lceil\sup_{0\leq r<n}\frac1{r+1}{\rm Card}\Big(\Big\{k:\ l_k\leq r\Big\}\Big)\Big\rceil.$$
It is sufficient to show the corresponding inequality for distribution functions:
$$d_{(a_{l_k})_{k=0}^{n-1}}(t)\leq C({\bf l})\cdot d_{(a_k)_{k=0}^{n-1}}(t).$$
Indeed, if $t\in(a_{r+1},a_r)$ for some $0\leq r<n,$ then
$$d_{(a_{l_k})_{k=0}^{n-1}}(t)=\sum_{a_{l_k}>t}1=\sum_{a_{l_k}\geq a_r}1={\rm Card}\Big(\Big\{k:\ l_k\leq r\Big\}\Big).$$
By definition of $C(l),$ we have
$$d_{(a_{l_k})_{k=0}^{n-1}}(t)\leq C({\bf l})\cdot (r+1)=C({\bf l})\cdot d_{(a_k)_{k=0}^{n-1}}(t).$$
This proves the claim of Step 3.

{\bf Step 4:} Combining Step 2 and Step 3, we obtain
$$\int_0^1\Big\|\fktvec\Big\|_E^pdt\leq\sum_{{\bf l}\in\Delta_n}\Big\|\sigma_{C({\bf l})}a\Big\|_E^p\cdot \prod_{k=0}^{n-1}p_{k,l_k}.$$
By Lemma \ref{trivial sigma lemma}, we have
$$\int_0^1\Big\|\fktvec\Big\|_E^pdt\leq C_E\cdot\sum_{{\bf l}\in\Delta_n}C({\bf l})^{p(1+\log_2(C_E))}\cdot \prod_{k=0}^{n-1}p_{k,l_k}\cdot\|a\|_E^p.$$
Since $C({\bf l})\geq1$ and since $\lceil c\rceil\leq c+1\leq 2c$ for all constants $c\geq1,$ it follows that
$$\int_0^1\Big\|\fktvec\Big\|_E^pdt\leq $$
$$\leq 2^pC_E^{p+1}\sum_{{\bf l}\in\Delta_n}\Big(\prod_{k=0}^{n-1}p_{k,l_k}\Big)\Big(\sup_{0\leq r<n}\frac1{r+1}{\rm Card}\Big(\Big\{k:\ l_k\leq r\Big\}\Big)\Big)^{p(1+\log_2(C_E))}\cdot\|a\|_E^p.$$

Setting $q=p(1+\log_2(C_E))$ and using Lemma \ref{lem-junge}, we obtain
$$\sum_{{\bf l}\in\Delta_n}\Big(\prod_{k=0}^{n-1}p_{k,l_k}\Big)\Big(\sup_{0\leq r<n}\frac1{r+1}{\rm Card}\Big(\Big\{k:\ l_k\leq r\Big\}\Big)\Big)^q\leq\Big(c_{abs}\frac{q}{\log(eq)}\Big)^q.$$
Therefore,
$$\int_0^1\Big\|\fktvec\Big\|_E^pdt\leq 2^pC_E^{p+1}\Big(c_{abs}\frac{q}{\log(eq)}\Big)^q\cdot \Big\|\Big(\mu(k,f)\Big)_{k=0}^{n-1}\Big\|_E^p,\quad f=\bigoplus_{k=0}^{n-1}f_k.$$
Taking $p$-th root, we complete the proof.
\end{proof}

For any $t\in(0,1)$, let
\begin{align}\label{marcin}
\psi(t):=\int_{0}^t\mu\left(s,K\chi_{(0,1)}\right)\,ds.
\end{align}
The \emph{Marcinkiewicz space $M_{\psi}$} is defined to be the space of all measurable functions $f$ on $(0,1)$ such that
$$\|f\|_{M_{\psi}}:=\sup_{t}\frac{\int_0^t\mu(s,f)\,ds}{\psi(t)}<\infty.$$

\begin{lemma}\label{K1 lemma} Let $\psi$ be as in \eqref{marcin}. We have
$$\|g\|_{M_{\psi}}\approx \sup_{p\geq1}\frac{\log(ep)}{p}\|g\|_p,\quad g\in M_{\psi}.$$
\end{lemma}
\begin{proof} By definition of a Marcinkiewicz space, we have $g\prec\prec\|g\|_{M_{\psi}}K\chi_{(0,1)}.$ Therefore,
$$\|g\|_p\leq\|g\|_{M_{\psi}}\|K\chi_{(0,1)}\|_p.$$
Since $K\chi_{(0,1)}$ is a Poisson random variable with parameter $1,$ it follows that
$$\|K\chi_{(0,1)}\|_p\approx\frac{p}{\log(ep)},\quad p\geq1.$$
Therefore,
$$\frac{\log(ep)}{p}\|g\|_p\lesssim \|g\|_{M_{\psi}},\quad p\geq1.$$
Taking the supremum over $p\geq1,$ we infer that $RHS\leq LHS.$

Conversely, let $\|g\|_{M_{\psi}}=1.$ Choose $t\in(0,\frac1e)$ such that
$$\int_0^t\mu(s,g)ds\geq \frac1{10}\psi(t).$$
Using the notion of Hardy-Littlewood submajorization (denoted by $\prec\prec$), we write
$$\frac{\psi(t)}{t}\chi_{(0,t)}\prec\prec 10g.$$
Choose $p=\log(\frac1t)$ (since $t\in(0,\frac1e),$ it follows that $p>1$). We have
$$10\frac{\log(ep)}{p}\|g\|_p\geq \frac{\log(ep)}{p}\cdot \frac{\psi(t)}{t}\cdot t^{\frac1p}=\frac{\log(e\log(\frac1t))}{\log(\frac1t)}\cdot \frac{\psi(t)}{t}\cdot e^{-1}.$$
By Lemma 4.3 in \cite{AS2005}, we have that
$$\psi(t)\approx\frac{t\log(e\log(\frac1t))}{\log(\frac1t)},\quad t\in(0,\frac1e).$$
Hence, for {\it chosen} $p,$
$$\frac{\log(ep)}{p}\|g\|_p\gtrsim 1.$$
In particular,
$$\sup_{p\geq1}\frac{\log(ep)}{p}\|g\|_p\gtrsim 1.$$
This proves $LHS\leq RHS.$
\end{proof}

\begin{lemma}\label{direct K lemma} Let $X$ be a quasi-Banach symmetric function space and let $r>0$ be such that the Kruglov operator $K:X_r\to X_r.$ We have
$$\|f\|_X\lesssim_{X,r}\sup_{p\geq1}\Big(\frac{\log(ep)}{p}\Big)^{\frac1r}\|f\|_p,\quad f\in X.$$
\end{lemma}
\begin{proof} Let $\psi$ be as in \eqref{marcin}. We have $K:X_r\to X_r.$ In particular, $K:L_{\infty}\to X_r$ and, therefore,
$K\chi_{(0,1)}\in X_r.$ Hence, $M_{\psi}\subset X_r$ and, therefore,
$$\|g^{\frac1r}\|_X^r\lesssim_{X,r}\|g\|_{M_{\psi}}.$$
Hence,
$$\|f\|_X\lesssim_{X,r}\|f^r\|_{M_{\psi}}^{\frac1r}.$$
By Lemma \ref{K1 lemma}, we have
$$\|g\|_{M_{\psi}}\approx\sup_{p\geq1}\frac{\log(ep)}p\|g\|_p \approx_r\sup_{p\geq\frac1r}\frac{\log(epr)}{pr}\|g\|_p.$$
Hence,
$$\|f\|_X\lesssim_{X,r}\Big(\sup_{p\geq\frac1r}\frac{\log(epr)}{pr}\|f^r\|_p\Big)^{\frac1r}=\sup_{p\geq\frac1r}\Big(\frac{\log(epr)}{pr}\Big)^{\frac1r}\|f\|_{pr}.$$
Renaming $pr$ into $p,$ we complete the proof.
\end{proof}

The following lemma provides an inverse estimate to \cite[Corollary 2]{Astashkin JUNGE}.

\begin{lemma}\label{tail lemma} Let $X$ be a quasi-Banach symmetric function space such that $K:X_r\to X_r$ for all $0<r\leq 1.$ Let $E$ be a symmetric quasi-Banach sequence space. Let $(f_k)_{k=0}^{n-1}\subset X$ be a sequence of independent random variables. We have
$$\big\|\|(f_k)_{k\geq0}\|_E\big\|_X\lesssim_{X,E}\big\|\big(\mu(k,f)\big)_{k\geq0}\big\|_E,\quad f=\bigoplus_{k=0}^{n-1}f_k.$$
\end{lemma}
\begin{proof} Let $r\in(0,1)$ be such that $\frac1r=1+\log_2(C_E).$ By Lemma \ref{direct K lemma}, we have
$$\big\|\|(f_k)_{k\geq0}\|_E\big\|_X\lesssim_{X,r}\sup_{p\geq1}(\frac{\log(ep)}{p})^{\frac1r}\big\|\|(f_k)_{k\geq0}\|_E\big\|_p.$$
Thus,
$$\big\|\|(f_k)_{k\geq0}\|_E\big\|_X\lesssim_{X,E}\sup_{p\geq1}(\frac{\log(ep)}{p})^{1+\log_2(C_E)}\big\|\|(f_k)_{k\geq0}\|_E\big\|_p.$$
By Lemma \ref{vector p-estimate}, we have
$$\Big\|\Big\|(f_k)_{k\geq0}\Big\|_E\Big\|_p\lesssim_E \Big(\frac{p}{\log(ep)}\Big)^{1+\log_2(C_E)}\cdot \Big\|\Big(\mu(k,f)\Big)_{k=0}^{n-1}\Big\|_E.$$
Combining these inequalities, we complete the proof.
\end{proof}

In order to show the following embedding lemma, we first introduce the notion of the $p$-norm.
Recall that a quasi-norm $\|\cdot\|$ is called a \emph{$p$-norm} ($p\in(0,1)$) if $\|f+g\|^p\le\|f\|^p+\|g\|^p$.
\begin{lemma}\label{16 lemma} Let $E$ be a quasi-Banach symmetric sequence space. Then there exists $p>0$ such that $\ell_p\subset E.$
\end{lemma}
\begin{proof} By Aoki-Rolewicz theorem (see \cite{Aoki1942}), $\|\cdot\|_E$ is equivalent to a $p_0$-norm. Let $p<p_0.$
We claim that $x_p=((k+1)^{-\frac1p})_{k\geq0}\in E.$ Indeed,
$$x_p=\sum_{k=0}^{\infty}(k+1)^{-\frac1p}e_k.$$
Thus, recalling that $\|e_k\|_E=1$ for all $k\ge 0$, we arrive at
$$\|x_p\|_E^{p_0}\lesssim_E \sum_{k=0}^{\infty}(k+1)^{-\frac{p_0}{p}}\|e_k\|_E^{p_0}=\sum_{k=1}^{\infty}k^{-\frac{p_0}{p}}<\infty.$$
This proves the claim.
Now, for every $x\in \ell_p,$ we have $\mu(x)\leq\|x\|_p\cdot x_p.$ Thus,
$$\|x\|_E=\|\mu(x)\|_E\leq\|x\|_p\|x_p\|_E.$$
In other words, $\ell_p\subset E.$
\end{proof}

\begin{lemma}\label{head lemma} Let $X$ be a quasi-Banach symmetric function space such that $K:X_r\to X_r$ for all $0<r\leq 1.$ Let $E$ be a symmetric quasi-Banach sequence space. Let $(f_k)_{k=0}^{n-1}\subset X$ be a sequence of independent random variables. If
$$\sum_{k=0}^{n-1}m({\rm supp}(f_k))\leq 1,$$
then
$$\big\|\|(f_k)_{k\geq0}\|_E\big\|_X\lesssim_{X,E}\|f\|_X,\mbox{ where }f=\bigoplus_{k=0}^{n-1}f_k.$$
\end{lemma}
\begin{proof} Choose $r$ so small that $\ell_r\subset E.$ We have
$$\big\|\|(f_k)_{k\geq0}\|_E\big\|_X\lesssim_E \big\|\|(f_k)_{k\geq0}\|_r\big\|_X=\big\|\sum_{k\geq 0}|f_k|^r\big\|_{X_r}^{\frac1r}.$$
Since $K:X_r\to X_r,$ it follows that the Johnson-Schechtman inequality is true in $X_r$ (see \cite{pacific}). Thus,
$$\big\|\sum_{k\geq 0}|f_k|^r\big\|_{X_r}\lesssim_{X_r}\big\|\bigoplus_{k\geq 0}|f_k|^r\big\|_{X_r}=\big\|\bigoplus_{k\geq 0}|f_k|\big\|^r_X.$$
A combination of these inequalities yields the assertion.
\end{proof}

\begin{lemma}\label{lower simple} Let $X$ be a quasi-Banach symmetric function space. Let $E$ be a symmetric quasi-Banach sequence space. Let $(f_k)_{k=0}^{n-1}\subset X$ be a sequence of independent random variables. We have
$$\big\|\|(f_k)_{k\geq0}\|_E\big\|_X\gtrsim_{X,E}\|\mu(f)\chi_{(0,1)}\|_X,\mbox{ where }f=\bigoplus_{k=0}^{n-1}f_k.$$
\end{lemma}
\begin{proof} We have $\|\cdot\|_E\gtrsim_E\|\cdot\|_{\infty}.$ In what follows, we assume without loss of generality that $E=\ell_{\infty}.$

Without loss of generality, $\mu(f)$ does not have intervals of constancy. Let
$$g_k=|f_k|\chi_{\{|f_k|>\mu(1,f)\}},\quad g=\bigoplus_{k=0}^{n-1}g_k.$$
By Lemma 3 in \cite{JS}, we have
$$\mu(\max_{0\leq k<n}g_k)\geq\sigma_{\frac12}\mu(g).$$
Hence,
$$\big\|\max_{0\leq k<n}|f_k|\big\|_X\geq\big\|\max_{0\leq k<n}|g_k|\big\|_X\gtrsim_X\|g\|_X=\|\mu(f)\chi_{(0,1)}\|_X.$$
\end{proof}

We are now ready to prove the main result in this section.

\begin{proof}[Proof of Theorem \ref{main-new}] Let us firstly prove the implication \eqref{mainnewb}$\Rightarrow$\eqref{mainnewa}. Without loss of generality, $\mu(f)$ does not have intervals of constancy. For $0\le k<n$ let
$$g_k=f_k\chi_{\{|f_k|>\mu(1,f)\}},\quad h_k=f_k-g_k.$$
Denote for brevity
$$g=\bigoplus_{k=0}^{n-1}g_k,\quad h=\bigoplus_{k=0}^{n-1}h_k.$$
Note that
$$\mu(g)=\mu(f)\chi_{(0,1)},\quad\mu(h)\leq\min\{\mu(f),\mu(1,f)\}.$$
By the triangle inequality we have
$$\big\|\|(f_k)_{k\geq0}\|_E\big\|_X\lesssim_{X,E}\big\|\|(g_k)_{k\geq0}\|_E\big\|_X+\big\|\|(h_k)_{k\geq0}\|_E\big\|_X.$$
By Lemma \ref{head lemma}, we have
$$\big\|\|(g_k)_{k\geq0}\|_E\big\|_X\lesssim_{X,E}\|g\|_X=\|\mu(f)\chi_{(0,1)}\|_X.$$
By Lemma \ref{tail lemma}, we have
$$\big\|\|(h_k)_{k\geq0}\|_E\big\|_X\lesssim_{X,E}\big\|\big(\mu(k,h)\big)_{k\geq0}\big\|_E\lesssim_E\big\|\big(\mu(k,f)\big)_{k\geq1}\big\|_E.$$
Combining these three inequalities, we obtain
$$\big\|\|(f_k)_{k\geq0}\|_E\big\|_X\lesssim_{X,E}\|\mu(f)\chi_{(0,1)}\|_X+\big\|\big(\mu(k,f)\big)_{k\geq1}\big\|_E.$$
This proves the upper estimate.

The lower estimate
$$\big\|\|(f_k)_{k\geq0}\|_E\big\|_X\gtrsim_{X,E}\|\mu(f)\chi_{(0,1)}\|_X$$
is established in Lemma \ref{lower simple}. The lower bound
$$\big\|\|(f_k)_{k\geq0}\|_E\big\|_X\gtrsim_{X,E}\big\|\big(\mu(k,f)\big)_{k\geq1}\big\|_E$$
is established\footnote{Note that our Lemma \ref{astashkin theorem1} below improves the key estimate in \cite{Astashkin JUNGE}.} in \cite{Astashkin JUNGE}. This completes the proof of the implication \eqref{mainnewb}$\Rightarrow$\eqref{mainnewa}.

Let us now prove the implication \eqref{mainnewa}$\Rightarrow$\eqref{mainnewb}. For every $0<p<1$ and for every sequence of nonnegative  independent random variables $(f_k)_{k=0}^{n-1}$ with
$$\sum_{k=0}^{n-1}m({\rm supp}(f_k))\le1,$$
take $E=\ell_p$ and $h_k=f_k^{\frac1p}$, $0\le k\le n-1$. Then by \eqref{mainnewa}, we have
$$\left\|\left\|(h_k)_{k=0}^{n-1}\right\|_{\ell_p}\right\|_X\lesssim_{X,p} \left\|\bigoplus_{k=0}^{n-1}h_k\right\|_X.$$
This implies that
$$\left\|\sum_{k=0}^{n-1}f_k\right\|_{X_p}\lesssim_{X_p} \left\|\bigoplus_{k=0}^{n-1}f_k\right\|_{X_p}.$$
By the assumption, $X$ has the Fatou property, and, hence, so does $X_p.$
Therefore, $X_p$ satisfies the conditions of Lemma \ref{lem-pjm}. Thus, $K$ is bounded on $X_p.$
\end{proof}

\begin{proof}[Proof of Corollary \ref{lplq corollary}] It follows from \cite{JS} and \cite{AS2005} that $K$ acts boundedly in every $L_p,\ 0<p<\infty$ and this fact now guarantees that \eqref{m} holds for $X=L_p,\ 0<p<\infty$ and $E=\ell_q,\ 0<q<\infty.$ That is, we have
$$\Big\|\ \Big\|\sum _{k\geq0}f_ke_k\Big\|_{\ell_q}\ \Big\|_{L_p}\approx_{p,q}\big\|\mu(f)\chi_{(0,1)}\big\|_p+\big\|\big(\mu(k,f)\big)_{k\geq1}\big\|_q.$$
However,
$$\big\|\mu(f)\chi_{(0,1)}\big\|_p+\big\|\big(\mu(k,f)\big)_{k\geq1}\big\|_q\approx\|f\|_{L_p+L_q}$$
for $p\leq q$ and
$$\big\|\mu(f)\chi_{(0,1)}\big\|_p+\big\|\big(\mu(k,f)\big)_{k\geq1}\big\|_q\approx\|f\|_{L_p\cap L_q}$$
for $p\geq q.$
\end{proof}

\section{Proofs of Theorem \ref{main modular thm}}

In this section, we provide the proof of Theorem \ref{main modular thm}. We first prove the upper estimate of \eqref{first main}. Observe that we do not impose on $E$ any additional restrictive assumptions as in \cite{Gordon}.

Applying Lemma \ref{vector p-estimate} and the idea from \cite{Jiao Sukochev and Zanin}, we prove the following key lemma.
\begin{lemma}\label{tail upper lemma}
Let $\Phi\in\Delta_2$ be an Orlicz function. Let $E$ be a quasi-Banach symmetric sequence space. For every $n\in\mathbb{N}$ and for every sequence $(f_k)_{k=0}^{n-1}\subset L_{\infty}(0,1)$ of independent random variables, the following inequality holds:
\begin{equation}\label{tail upper estimate}
\int_0^1\Phi\Big(\Big\|\fktvec\Big\|_E\Big)dt\lesssim_{E,\Phi}\Phi\Big(\Big\|\Big(\mu(k,f)\Big)_{k=0}^{n-1}\Big\|_E\Big),\quad f=\bigoplus_{k=0}^{n-1}f_k.
\end{equation}
\end{lemma}
\begin{proof} Since $\Phi\in\Delta_2,$ $\Phi$ is (equivalent to) a $1$-convex and $q$-concave Orlicz function for some $1\leq q<\infty.$
Since $\Phi$ is $1$-convex and $q$-concave, it follows that the mappings
$$t\to\frac{\Phi(t)}{t},\quad t\to\frac{t^q}{\Phi(t)}$$
are increasing (see, for instance, \cite[Lemma 6]{AS2014JMAA}). Define the function $\phi: (0,\infty)\rightarrow (0,\infty)$ by setting $\Phi(t)=t\phi(t^{q-1}),$ $t>0.$
Obviously, the mappings
$$t\to \phi(t),\quad t\to\frac{t}{\phi(t)}$$
are increasing. In other words, $\phi$ is quasi-concave. Using \cite[Lemma 5.4.3]{Bergh and Lofstorm}, we have
$$\phi(t)\approx\alpha+\beta t+\int_0^\infty \min\{\tau,t\}dm_0(\tau),$$
where $\alpha\geq0,\beta\geq0,$ and $m_0$ is an increasing function bounded from above and with $\lim_{t\rightarrow0}tm_0(t)=0.$
Hence,
$$\Phi(u)\approx \Phi_0(u)+\Phi_1(u)+\Phi_2(u),\quad u>0,$$
where $\Phi_1(u)=u,$ $\Phi_2(u)=u^q,$ $u>0,$ and where
$$\Phi_0(u)=\int_0^{\infty}\min\{u^q,\tau u\}dm_0(\tau)=\int_0^{\infty}\min\{(su)^q,su\}d\nu(s).$$
In the last inequality, we made a substitution $\tau=s^{1-q}$ and defined the positive measure $\nu$ on $(0,\infty)$ by setting
$$d\nu(s)=-s^{-q}dm_0(s^{1-q}).$$
By Lemma \ref{vector p-estimate}, we have
\begin{eqnarray*}
&\quad&\int_0^1\min\Big\{s\Big\|\fktvec\Big\|_E,\,\,\,s^q\Big\|\fktvec\Big\|_E^q\Big\}dt\\&\leq&\min\Big\{\int_0^1s\Big\|\fktvec\Big\|_Edt,\,\,\, \int_0^1s^q\Big\|\fktvec\Big\|_E^qdt\Big\}\\
&\lesssim_{E,q}&\min\Big\{s\Big\|\Big(\mu(k,f)\Big)_{k=0}^{n-1}\Big\|_E,\,\,\, s^q\Big\|\Big(\mu(k,f)\Big)_{k=0}^{n-1}\Big\|_E^q\Big\}.
\end{eqnarray*}
Integrating over $s$ with respect to the measure $\nu,$
we obtain \eqref{tail upper estimate} for the Orlicz function $\Phi_0.$
The inequality \eqref{tail upper estimate} holds for Orlicz functions $\Phi_1$ and $\Phi_2$
by Lemma \ref{vector p-estimate}. Summing these inequalities, we complete the proof.
\end{proof}

The assertion below is proved in \cite[Lemma 3.3]{JSXZ} in the special case when $\Psi$ is an Orlicz function.

\begin{lemma}\label{K Psi estimate} Let $\Psi$ be an increasing function on $[0,\infty)$ such that
$$\Psi(2t)\leq c_{\Psi}\Psi(t),\quad t>0,$$
where $c_{\Psi}$ is a positive constant depending only on $\Psi$. Let $K$ be the Kruglov operator. For every positive $f\in L_{\Psi}(0,1),$  we have
$$\int_0^1\Psi((Kf)(t))dt\lesssim_{\Psi}\int_0^1\Psi(f(t))dt.$$

\end{lemma}
\begin{proof} By \eqref{K def}, we have
$$\int_0^1\Psi((Kf)(t))dt=\sum_{m=1}^{\infty}\frac1{e\cdot m!}\int_{[0,1]^m}\Psi(\sum_{k=1}^mf(\omega_k))d\omega_1\cdots d\omega_k.$$
Since $\Psi\in\Delta_2$ is increasing and $u_1+u_2\le\max\{2u_1,2u_2\}$, we have
$$\Psi(u_1+u_2)\le \Psi(2u_1)+\Psi(2u_2)\le C_{\Psi}(\Psi(u_1)+\Psi(u_2)).$$
Denote $c_{\psi}=2\log(C_{\psi}).$ By induction, we have
$$\Psi(\sum_{k=1}^{2^n}u_k)\leq C_{\psi}^n\cdot\sum_{k=1}^{2^n}\Psi(u_k).$$
For $m\geq 2$, choose $n$ such that $m\in[2^n,2^{n+1}).$ We have
$$\Psi(\sum_{k=1}^mu_k)\leq C_{\psi}^{n+1}\sum_{k=1}^m\Psi(u_k)\leq m^{c_{\psi}}\sum_{k=1}^m\Psi(u_k).$$
Thus,
$$\int_{[0,1]^m}\Psi(\sum_{k=1}^mf(\omega_k))d\omega_1\cdots d\omega_k\leq m^{c_{\psi}}\sum_{k=1}^m\int_{[0,1]^m}\Psi(f(\omega_k))d\omega_1\cdots d\omega_k$$
and
$$\int_0^1\Psi(\mu(t,Kf))dt\leq\sum_{m=1}^{\infty}\frac{m^{c_{\Psi}+1}}{e\cdot m!}\cdot\int_0^1\Psi(f(t))dt.$$
\end{proof}

The observation below is trivial but quite useful.

\begin{lemma}\label{14 lemma} Let $(f_k)_{k=0}^{n-1}\subset L_1(0,1)$ and $\{g_k\}_{k=0}^{n-1}\subset L_1(0,1)$
be sequences of independent positive functions. If $\mu(g_k)\leq\mu(f_k)$ for $0\leq k<n,$ then
$$\mu\Big(\sum_{k=0}^{n-1}g_k\Big)\leq\mu\Big(\sum_{k=0}^{n-1}f_k\Big).$$
\end{lemma}
\begin{proof} By assumption, the function $\sum_{k=0}^{n-1}f_k$ is equimeasurable with the function
$$\omega \mapsto \sum_{k=0}^{n-1}\mu(\omega_k,f_k).$$
Similarly, the function $\sum_{k=0}^{n-1}g_k$ is equimeasurable with the function
$$\omega \mapsto \sum_{k=0}^{n-1}\mu(\omega_k,g_k).$$
It is immediate that
$$0\leq\sum_{k=0}^{n-1}\mu(\omega_k,g_k)\leq\sum_{k=0}^{n-1}\mu(\omega_k,f_k).$$
This completes the proof.
\end{proof}

The following lemma improves \cite[Lemma 8]{pacific}.

\begin{lemma}\label{15 lemma} For every $n\in\mathbb{N}$ and for every sequence
$(f_k)_{k=0}^{n-1}\subset L_1(0,1)$ of independent positive functions satisfying the condition
$$\sum_{k=0}^{n-1}m({\rm supp}(f_k))\leq1,$$
the following inequality holds:
$$\mu\Big(\sum_{k=0}^{n-1}f_k\Big)\leq 3\sigma_3\mu(Kf),\quad f=\bigoplus_{k=0}^{n-1}f_k.$$
\end{lemma}
\begin{proof} Since the sequence $(f_k)_{k=0}^{n-1}$ consists of independent random variables, it follows that $\sum_{k=0}^{n-1}f_k$ is equimeasurable with the function
$$F:\omega \mapsto  \sum_{k=0}^{n-1}f_k(\omega_k).$$
Consider the functions $F_m,$ $m=0,1,2$ defined by the formula
$$F_m:\omega \mapsto \sum_{k=0}^{n-1}(\sigma_{\frac13}f_k)(\omega_k+\frac{m}{3}{\rm mod}1).$$
Clearly, these functions are equimeasurable and $F$ is equimeasurable with $F_0+F_1+F_2.$ Therefore,
$$\mu(F)\leq\sigma_3\mu(F_0)+\sigma_3\mu(F_1)+\sigma_3\mu(F_3)=3\sigma_3\mu(F_0).$$
Let $(h_k)_{k=0}^{n-1}\subset L_1(0,1)$ be a sequence of independent copies of $(Kf_k)_{k=0}^{n-1}\subset L_1(0,1).$ Since $\sigma_{\frac13}\mu(f_k)\leq\mu(Kf_k),$ it follows from Lemma \ref{14 lemma} that
$$\mu(F_0)\leq\mu\Big(\sum_{k=0}^{n-1}h_k\Big).$$
However, the function $\sum_{k=0}^{n-1}h_k$ is equimeasurable with $Kf$ and the assertion follows.
\end{proof}

Note that the lemma below extends \cite[Lemma 3.4]{JSXZ}.

\begin{lemma}\label{head upper lemma} Let $\Phi\in\Delta_2$ be an Orlicz function. For every $n\in\mathbb{N},$ $0<p<\infty$ and for every sequence
$(f_k)_{k=0}^{n-1}\subset L_{\Phi}(0,1)$ of independent random variables satisfying the condition
$$\sum_{k=0}^{n-1}m({\rm supp}(f_k))\leq 1,$$
the following inequality holds:
\begin{equation}\label{head upper estimate}
\int_0^1\Phi\Big(\Big\|\fktvec\Big\|_p\Big)dt\lesssim_{\Phi,p}\int_0^1\Phi(f(t))dt,\quad f=\bigoplus_{k=0}^{n-1}f_k.
\end{equation}
\end{lemma}
\begin{proof} Without loss of generality, $f_k\geq0.$ Set $g_k=f_k^p,$ $0\leq k<n,$ and
$$g=\bigoplus_{k=0}^{n-1}g_k.$$
Observe that $g$ is supported on $[0,1].$ Also, set $\Psi(t)=\Phi(t^{\frac1p}),$ $t>0,$ then $\Psi\in\Delta_2.$ Indeed, choose $n\geq0$ such that $\frac1p\leq n,$ the
$$\Psi(2t)=\Phi(2^{\frac1p}t^{\frac1p})\leq\Phi(2^nt^{\frac1p})\leq C_{\Phi}^n\Phi(t^{\frac1p})=C_{\Phi}^n\Psi(t).$$
That is,
$$\Psi(2t)\leq c_{\Psi}\Psi(t),\quad t>0.$$
Clearly,
$$\int_0^1\Phi\Big(\Big\|\fktvec\Big\|_p\Big)dt=\int_0^1\Psi(\sum_{k=0}^{n-1}g_k(t))dt.$$
By Lemma \ref{15 lemma}, we have
$$\mu\Big(\sum_{k=0}^{n-1}g_k\Big)\leq 3\sigma_3\mu(Kg).$$
Thus,
\begin{eqnarray*}
\int_0^1\Phi\Big(\Big\|\fktvec\Big\|_p\Big)dt
&\leq& \int_0^1\Psi\left(3\mu\left(\frac{t}{3},Kg\right)\right)dt\leq 3\int_0^1\Psi(3(Kg)(t))dt\\
&\stackrel{L.\ref{K Psi estimate}}{\lesssim_{\Psi}}&\int_0^1\Psi(g(t))dt=\int_0^1\Phi(f(t))dt.
\end{eqnarray*}
The proof of Lemma \ref{head upper lemma} is completed.
\end{proof}

\begin{proof}[Proof of the upper estimate in Theorem \ref{main modular thm}] Without loss of generality, $\mu(f)$ does not have intervals of constancy. Define new random variables
$$g_k=f_k\chi_{\{f_k>\mu(1,f)\}},\quad h_k=f_k-g_k,\quad 0\leq k<n.$$
The random variables $g_k,$ $0\leq k<n,$ are positive and independent and so are the random variables $h_k,$ $0\leq k<n.$ Denote for brevity
$$g=\bigoplus_{k=0}^{n-1}g_k,\quad h=\bigoplus_{k=0}^{n-1}h_k.$$
We have
$$\mu(g)=\mu(f)\chi_{(0,1)},\quad \mu(h)\leq\min\{\mu(f),\mu(1,f)\}.$$
Since $\Phi\in\Delta_2,$ it follows that
$$\Phi(t+s)\leq c_{\Phi}\cdot(\Phi(t)+\Phi(s)),\quad t,s>0.$$
Hence,
$$\int_0^1\Phi\Big(\Big\|\fktvec\Big\|_E\Big)dt\lesssim\int_0^1\Phi\Big(\Big\|\gktvec\Big\|_E\Big)dt+\int_0^1\Phi\Big(\Big\|\hktvec\Big\|_E\Big)dt.$$
Using Lemma \ref{16 lemma}, choose $p$ so small that $\ell_p\subset E.$ It follows that
$$\int_0^1\Phi\Big(\Big\|\gktvec\Big\|_E\Big)dt\lesssim_{E,\Phi}\int_0^1\Phi\Big(\Big\|\gktvec\Big\|_p\Big)dt.$$
Using Lemma \ref{head upper lemma}, we infer that
$$\int_0^1\Phi\Big(\Big\|\gktvec\Big\|_E\Big)dt\lesssim_{\Phi}\int_0^1\Phi(\mu(t,f))dt.$$
On the other hand, it follows from Lemma \ref{tail upper estimate} that
$$\int_0^1\Phi\Big(\Big\|\hktvec\Big\|_E\Big)dt\lesssim_{E,\Phi}\Phi\Big(\Big\|\Big(\mu(k,h)\Big)_{k=0}^{n-1}\Big\|_E\Big).$$
Since $\Phi$ is increasing, it follows that
$$\int_0^1\Phi\Big(\Big\|\hktvec\Big\|_E\Big)dt\lesssim_{E,\Phi}\Phi\Big(\Big\|\Big(\mu(k,f)\Big)_{k=1}^n\Big\|_E\Big).$$
Combining the estimates above, we conclude the proof.
\end{proof}

We now prepare some background for the proof of the lower estimate in Theorem \ref{main modular thm}. We begin with
Lemma \ref{astashkin theorem1}, which improves on \cite[Theorem 1]{Astashkin JUNGE}. Our proof is significantly simpler than that of \cite[Theorem 1]{Astashkin JUNGE}, even though it also uses \cite[Proposition 1]{Astashkin JUNGE} in a crucial way.

To state Lemma \ref{astashkin theorem1}, let $x_1\geq x_2\geq\cdots\geq x_n\geq 0.$
Let $\xi_k:(0,1)\to\{x_1,\cdots,x_n\},$ $1\leq k\leq n,$ be independent random variables such that
$$\sum_{k=1}^n m(\{t\in(0,1):\xi_k(t)=x_j\})=1\quad (j=1,\ldots,n).$$
Define functions $\eta_k,$ $1\leq k\leq n,$ as follows: for a fixed $t\in(0,1),$ we set
$$(\eta_k(t))_{k=1}^{n}=\mu((\xi_k(t))_{k=1}^{n}).$$

\begin{lemma}\label{astashkin theorem1} We have
$$m\Big(\Big\{t:\ \eta_k(t)\geq x_{4k-3},\quad \forall\, 1\leq k\leq\Big\lfloor\frac{n+3}{4}\Big\rfloor\Big\}\Big)>\frac1{10}.$$
\end{lemma}
\begin{proof} Clearly,
$$m\Big(\Big\{t:\ \eta_k(t)\geq x_{4k-3},\quad \forall\, 1\leq k\leq\Big\lfloor\frac{n+3}{4}\Big\rfloor\Big\}\Big)\geq$$
$$\geq 1-\sum_{k=1}^{\lfloor\frac{n+3}{4}\rfloor}m\Big(\Big\{t:\ \eta_k(t)<x_{4k-3}\Big\}\Big).$$
By \cite[Proposition 1]{Astashkin JUNGE}, we have
$$m\Big(\Big\{t:\ \eta_k(t)<x_{4k-3}\Big\}\Big)\leq\sum_{l=n-k+1}^n 2^{l-(n-k+1)}\binom{n}{l}(\frac{n-4k+3}{n})^l(\frac{4k-3}{n})^{n-l}.$$
Thus,
$$m\Big(\Big\{t:\ \eta_k(t)\geq x_{4k-3}\quad \forall\, 0\leq k\leq\Big\lfloor\frac{n+3}{4}\Big\rfloor\Big\}\Big)\geq$$
$$\geq 1-\sum_{k=1}^{\lfloor\frac{n+3}{4}\rfloor}\sum_{l=n-k+1}^n 2^{l-(n-k+1)}\binom{n}{l}(\frac{n-4k+3}{n})^l(\frac{4k-3}{n})^{n-l}.$$
Denote
$$P:=\sum_{k=1}^{\lfloor\frac{n+3}{4}\rfloor}\sum_{l=n-k+1}^n 2^{l-(n-k+1)}\binom{n}{l}(\frac{n-4k+3}{n})^l(\frac{4k-3}{n})^{n-l}.$$
Setting $m=n-l$ and $j=k-1,$ we rewrite the sum for $P$ as follows:
$$P=\sum_{j=0}^{\lfloor\frac{n+3}{4}\rfloor-1}\sum_{m=0}^j 2^{j-m}\binom{n}{m}(\frac{n-4j-1}{n})^{n-m}(\frac{4j+1}{n})^m.$$
We claim that
$$2^{-m}\binom{n}{m}(\frac{n-4j-1}{n})^{n-m}(\frac{4j+1}{n})^m=(1-\frac{4j+1}{2n})^n\cdot \binom{n}{m}r^m(1-r)^{n-m},$$
where
$$r=\frac{4j+1}{2n-(4j+1)}\in(0,1].$$
Indeed,
$$2^{-m}(\frac{n-4j-1}{n})^{n-m}(\frac{4j+1}{n})^m=p^{n-m}q^m=$$
$$=(\frac{p}{p+q})^{n-m}(\frac{q}{p+q})^m(p+q)^n=r^m(1-r)^{n-m}(1-\frac{4j+1}{2n})^n,$$
where
$$p=\frac{n-4j-1}{n},\quad q=\frac{4j+1}{2n}.$$
It follows from the binomial formula that
$$P\leq\sum_{j=0}^{\lfloor\frac{n+3}{4}\rfloor-1}2^j\cdot (1-\frac{4j+1}{2n})^n.$$
Set $y=\frac{4j+1}{2n}\in[0,1].$ We have
$$1-y\leq e^{-y}$$
and, therefore,
$$(1-\frac{4j+1}{2n})^n=(1-y)^n\leq e^{-ny}=e^{-\frac{4j+1}{2}}.$$
Thus,
$$P\leq\sum_{j=0}^{\lfloor\frac{n+3}{4}\rfloor-1}2^j\cdot e^{-\frac{4j+1}{2}}\leq\frac{e^{-\frac12}}{1-2\cdot e^{-2}}<1-\frac1{10},$$
and the proof is completed.
\end{proof}

Our next lemma estimates the tail part of the right hand side in \eqref{first main} from the above. Its proof borrows some ideas from the proof of \cite[Corollary 2]{Astashkin JUNGE}.

\begin{lemma}\label{tail lower estimate}
Let $\Phi\in\Delta_2$ be an Orlicz function. Let $E$ be a symmetric quasi-Banach sequence space.
For every $n\in\mathbb{N}$ and for every sequence
$(f_k)_{k=0}^{n-1} \subset L_{\Phi}(0,1)$ of independent random variables, the following inequality holds:
\begin{equation}
\int_0^1\Phi\Big(\Big\|\fktvec\Big\|_E\Big)dt\gtrsim_{\Phi}\Phi\Big(\Big\|\Big(\mu(k,f)\Big)_{k=1}^n\Big\|_E\Big),\quad f=\bigoplus_{k=0}^{n-1}f_k.
\end{equation}
\end{lemma}
\begin{proof} Denote, for brevity, $x_l=\mu(l,f),$ $0\leq l\leq n.$  Without loss of generality, we may assume that the function $\mu(f)$ does not have any intervals at which its value is constant. Set
$$\xi_k=\sum_{l=1}^nx_l\chi_{(x_l,x_{l-1})}(f_k),\quad 1\leq k\leq n.$$
Since $\mu(f)$ has no intervals of constancy, it follows from the definition of $\xi_k$ that the inequality $f_k\geq\xi_k$ holds almost everywhere.

Let
$$A=\Big\{t:\ \eta_k(t)\geq x_{4k-3}\quad \forall\, 1\leq k\leq\Big\lfloor\frac{n+3}{4}\Big\rfloor\Big\},$$
where $(\eta_k(t))_{k=1}^{n}=\mu((\xi_k(t))_{k=1}^{n}).$
For every $t\in A,$ we have
$$\Phi\Big(\Big\|\fktvec\Big\|_E\Big)\geq\Phi\Big(\Big\|\Big(\xi_k(t)\Big)_{k=1}^n\Big\|_E\Big)\geq \Phi\Big(\Big\|\Big(x_{4k-3}\Big)_{k=1}^{\lfloor\frac{n+3}{4}\rfloor}\Big\|_E\Big).$$
By Lemma \ref{astashkin theorem1}, we have $m(A)>\frac1{10}.$ It follows that
\begin{equation}\label{lower estimate E}
LHS\geq \int_0^1\Phi\Big(\Big\|\fktvec\Big\|_E\Big)\chi_A(t)dt
\geq\frac1{10}\Phi\Big(\Big\|\Big(x_{4k-3}\Big)_{k=1}^{\lfloor\frac{n+3}{4}\rfloor}\Big\|_E\Big).
\end{equation}

We have
$$\left\|\left(x_k\right)_{k=1}^n\right\|_E
\leq\left\|\sum_{h=1}^4\sum_{j=1}^{\lfloor\frac{n+3}{4}\rfloor}x_{4j-4+h}e_{4j-4+h}\right\|_E
\leq $$
$$\leq C_E^2\cdot\sum_{h=1}^4\left\|\sum_{j=1}^{\lfloor\frac{n+3}{4}\rfloor}x_{4j-4+h}e_{4j-4+h}\right\|_E\leq
4C_E^2\left\|\sum_{j=1}^{\lfloor\frac{n+3}{4}\rfloor}x_{4j-3}e_{j}\right\|_E.$$
Since $\Phi\in\Delta_2,$ it follows that
$$\Phi(4C_E^2t)\leq c_{\Phi,E}\Phi(t),\quad t>0,$$
for some constant $c_{\Phi,E}.$ Therefore,
$$\Phi\Big(\Big\|\big(x_k\big)_{k=1}^n\Big\|_E\Big)\lesssim_{\Phi,E}
\Phi\Big(\Big\|\Big(x_{4k-3}\Big)_{k=1}^{\lfloor\frac{n+3}{4}\rfloor}\Big\|_E\Big).$$
Combining this and inequality \eqref{lower estimate E}, we complete the proof.
\end{proof}

The following assertion is due to Johnson and Schechtman (see \cite[Lemma 3]{JS}).

\begin{lemma}\label{distribution}
Let $(g_k)_{k=0}^{n-1}$ be a sequence of non-negative independent random variables defined on $(0,1)$ such that
$$\sum_{k=0}^{n-1}m({\rm supp}(g_k))\leq 1.$$
If $g$ is the corresponding disjoint sum, then
$$\mu(g)\leq\sigma_2\mu(\max_{0\leq k<n}g_k).$$
\end{lemma}

Now we are in a position to prove the lower estimate of \eqref{first main}.

\begin{proof}[Proof of the lower estimate in Theorem \ref{main modular thm}] Without loss of generality, $\mu(f)$ does not have intervals of constancy. For $0\le k<n$ we set
$$g_k=f_k\chi_{(\mu(1,f),\infty)}(f_k),\quad g=\bigoplus_{k=0}^{n-1}g_k.$$
Consequently,
$$\int_0^1\Phi\Big(\Big\|\fktvec\Big\|_E\Big)dt\geq \int_0^1\Phi\Big(\Big\|\gktvec\Big\|_E\Big)dt.$$
Since $E\subset \ell_{\infty},$ it follows that
$$\int_0^1\Phi\Big(\Big\|\fktvec\Big\|_E\Big)dt\geq\int_0^1\Phi\Big(\Big\|\gktvec\Big\|_{\infty}\Big)dt=\int_0^1\Phi\Big(\max_{0\leq k<n}|g_k(t)|\Big)dt.$$
It follows from Lemma \ref{distribution} that
$$\int_0^1\Phi\Big(\Big\|\fktvec\Big\|_E\Big)dt\geq \frac12\int_0^1\Phi(\mu(t,g))dt.$$
Clearly, $\mu(g)=\mu(f)\chi_{(0,1)}.$ Thus,
$$\int_0^1\Phi\Big(\Big\|\fktvec\Big\|_E\Big)dt\geq\frac12\int_0^1\Phi(\mu(t,f))dt.$$
The assertion follows now by combining the latter estimate with Lemma \ref{tail lower estimate}.
\end{proof}

\end{document}